\documentclass[a4paper,12pt]{article}
\usepackage{a4wide}
\usepackage{amsmath}
\usepackage{amssymb}
\usepackage{amsthm}
\usepackage{latexsym}
\usepackage{graphicx}
\usepackage[english]{babel}
\usepackage{makeidx}

\newtheorem{obs} [subsection]{Remark}

\newtheorem{prop}[subsection]{Proposition}

\newtheorem{teor}[subsection]{Theorem}
\newtheorem*{teor*}{Theorem}
\newtheorem{lema}[subsection]{Lemma}
\newtheorem{cor} [subsection]{Corollary}

\newcommand{\pa}{p_{\mathbf a}}

\newcommand{\Pa}{P_{\mathbf a}}
\newcommand{\za}{\zeta_{\mathbf a}}

\newcommand{\Res}{Res}

\newcommand{\stir}{\genfrac{[}{]}{0pt}{}}

\def\gcd{\operatorname{gcd}}

\begin{document}
\selectlanguage{english}
\frenchspacing
\numberwithin{equation}{section}

\large
\begin{center}
\textbf{On the restricted partition function}

Mircea Cimpoea\c s and Florin Nicolae
\end{center}
\normalsize

\begin{abstract}
For a vector $\mathbf a=(a_1,\ldots,a_r)$ of positive integers we prove formulas for
the restricted partition function $\pa(n): = $ the number of integer solutions $(x_1,\dots,x_r)$ to
$\sum_{j=1}^r a_jx_j=n$ with $x_1\geq 0, \ldots, x_r\geq 0$ and its polynomial part.
 

\end{abstract}

\section*{Introduction}

\footnotetext[1]{The first author was supported 
by a grant of the Romanian National Authority for Scientific
Research, CNCS - UEFISCDI, project number PN-II-ID-PCE-2011-1023.}

Let $\mathbf a:=(a_1,a_2,\ldots,a_r)$ be a sequence of positive integers, $r\geq 1$. 
The \emph{restricted partition function} associated to $\mathbf a$ is $\pa:\mathbb N \rightarrow \mathbb N$, 
$\pa(n):=$ the number of integer solutions $(x_1,\ldots,x_r)$ of $\sum_{i=1}^r a_ix_i=n$ with $x_i\geq 0$.


Sylvester \cite{sylvester},\cite{sylv} decomposed the restricted partition function in a sum of ``waves'', 
$$ \pa(n)=\sum_{j\geq 1} W_j(n,\mathbf a),$$
where the sum is taken over all distinct divisors $j$ of the components of $\mathbf a$ and showed that for each such $j$, 
$W_j(n,\mathbf a)$ is the coefficient of $t^{-1}$ in
$$ \sum_{0 \leq \nu <j,\; \gcd(\nu,j)=1 } \frac{\rho_j^{-\nu n} e^{nt}}{(1-\rho_j^{\nu a_1}e^{-a_1t})\cdots (1-\rho_j^{\nu a_r}e^{-a_rt}) },$$
where $\rho_j=e^{2\pi i/j}$ and $\gcd(0,0)=1$ by convention. 
Glaisher \cite{glaisher} made computations of the waves in particular cases.
Sills and Zeilberger \cite{sils} described an algorithm that provides explicit expressions for the number of
partitions of $n$ into at most $m$ parts.
Fel and Rubinstein \cite{rubfel} proved formulas for the Sylvester waves using Bernoulli and Euler polynomials of higher order.
Rubinstein \cite{rub} showed that all Sylvester waves can be expressed in terms of Bernoulli 
polynomials only. Bayad and Beck \cite[Theorem 3.1]{babeck} proved an explicit expression 
of the partition function $\pa(n)$ in terms of Bernoulli-Barnes polynomials and the Fourier Dedekind sums, in the case that $a_1,\ldots,a_r$
are pairwise coprime. 
Dilcher and Vignat \cite[Theorem 1.1]{dilcher} proved an explicit formula for $W_1(n,\mathbf a)$, which is the polynomial part of $\pa(n)$.

\indent
Let $D$ be a common multiple of $a_1,a_2,\ldots,a_r$. 
Bell \cite{bell} has proved that $\pa(n)$ is a quasi-polynomial of degree $r-1$, with the period $D$, i.e. 
$$\pa(n)=d_{\mathbf a,r-1}(n)n^{r-1}+\cdots+d_{\mathbf a,1}(n)n +d_{\mathbf a,0}(n),$$
where $d_{\mathbf a,m}(n+D)=d_{\mathbf a,m}(n)$ for $0\leq m\leq r-1$ and $n\geq 0$, and $d_{\mathbf a,r-1}(n)$ is not identically zero.
Our main result is the following formula for $\pa(n)$ which is stated in Theorem $1.8$:{\small
$$\pa(n) = \frac{1}{(r-1)!} \sum_{m=0}^{n-1} \sum_{\substack{0\leq j_1\leq \frac{D}{a_1}-1,\ldots, 0\leq j_r\leq \frac{D}{a_r}-1 \\ 
a_1j_1+\cdots+a_rj_r \equiv n (\bmod D)}} 
\sum_{k=m}^{r-1} \stir{r}{k} (-1)^{k-m} \binom{k}{m} D^{-k} (a_1j_1 + \cdots + a_rj_r)^{k-m}  n^m,$$}
where $\stir{r}{k}$ are the unsigned Stirling numbers of the first kind.
The novelty of this formula is that we reduce the computation of the restricted partition function $\pa(n)$ to solving the linear congruence
$a_1j_1+\cdots+a_rj_r \equiv n (\bmod D)$ in the restricted range $0\leq j_1\leq \frac{D}{a_1}-1,\ldots, 0\leq j_r\leq \frac{D}{a_r}-1$.


In Corollary $1.10$ we prove that 
$$ \pa(n) = \frac{1}{(r-1)!} \sum_{\substack{0\leq j_1\leq \frac{D}{a_1}-1,\ldots, 0\leq j_r\leq \frac{D}{a_r}-1 \\ a_1j_1+\cdots+a_rj_r \equiv n (\bmod D)}} 
\prod_{\ell=1}^{r-1} (\frac{n-a_1j_{1}- \cdots -a_rj_r}{D}+\ell ).$$
In Remark $1.11$ we present an algorithm to compute $\pa(n)$.

In Corollary $1.12$ we prove that $\pa(n)=0$ if and only if $n<a_1j_1+ \cdots + a_rj_r$ for all $(j_1,\ldots,j_r)$ such that 
$0\leq j_1\leq \frac{D}{a_1}-1,\ldots, 0\leq j_r\leq \frac{D}{a_r}-1$ and $a_1j_1+\cdots+a_rj_r \equiv n (\bmod D)$.

Our method is based on properties of the \emph{Barnes zeta function} associated to $\mathbf a$ and $w>0$
$$
\za(s,w):=\sum_{u_1,\ldots,u_r\geq 0}\frac{1}{(a_1u_1+\cdots+a_ru_r+w)^s},\; Re(s)>r,
$$
which is related to $\pa(n)$ by the formula
$$ \za(s,w)= \sum_{n\geq 0}\frac{\pa(n)}{(n+w)^s}.$$

As an illustration of our method, we provide new proofs for several results which are known in the literature.
For $r=2$ we prove in Corollary $1.13$ the classical formula of Popoviciu \cite{popoviciu} for $\pa(n)$ and in 
Corollary $1.14$ the well known formula $F(a_1,a_2)=a_1a_2-a_1-a_2$ for the Frobenius number of the coprime numbers $a_1$ and $a_2$.

For the polynomial part $\Pa(n)$ of $\pa(n)$ we prove in Corollary $2.6$ that
$$ \Pa(n) = \frac{1}{D(r-1)!} \sum_{0\leq j_1\leq \frac{D}{a_1}-1,\ldots, 0\leq j_r\leq \frac{D}{a_r}-1} \prod_{\ell=1}^{r-1} (\frac{n-a_1j_{1}- \cdots -a_rj_r}{D}+\ell ).$$
This extends Theorem $1.1$ of Dilcher and Vignat \cite{dilcher}: the authors prove the formula for $D=a_1\cdots a_r$ and we prove it for any
common multiple of $a_1,\ldots,a_r$.

The Dirichlet series associated to $\pa(n)$ is $\za(s):=\sum_{n=1}^{\infty}\frac{\pa(n)}{n^s}$. This converges for $Re(s)>r$ and defines
a meromorphic function in the whole complex plane with poles at most in the set $\{1,\ldots,r\}$ which are simple with residues 
$R_m:=Res_{s=m}\za(s)$. In Theorem $2.10$ we prove that 
$$R_m = \frac{(-1)^{r-m}}{(m-1)!}B_{r-m}(a_1,\ldots,a_r),
\; 1\leq m \leq r,$$ where $B_{r-m}(a_1,\ldots,a_r)$ are the Bernoulli Barnes numbers. This result can be obtained also as a 
direct consequence of formula $(3.9)$ in Ruijsenaars \cite{rui}. 

In Corollary $2.11$ we prove  
the result of Beck, Gessler and Komatsu \cite{beck} that the polynomial part of $\pa(n)$ is
$$P_{\mathbf a}(n) = \frac{1}{a_1\cdots a_r}\sum_{u=0}^{r-1}\frac{(-1)^u}{(r-1-u)!}\sum_{i_1+\cdots+i_r=u} 
\frac{B_{i_1}\cdots B_{i_r}}{i_1!\cdots i_r!}a_1^{i_1}\cdots a_r^{i_r} n^{r-1-u}.$$
These authors do not use the residues of the Barnes zeta function, as we are doing in the proof of Corollary $2.11$.

\section{A formula for the restricted partition function}

Let $\mathbf a:=(a_1,a_2,\ldots,a_r)$ be a sequence of positive integers, $r\geq 1$.
We use the following notations:

\begin{itemize}
\item $D$ is a common multiple of the numbers $a_1,a_2,\ldots,a_r$.

\item For $0\leq j_1\leq \frac{D}{a_1}-1$, $0\leq j_2 \leq \frac{D}{a_2}-1, \ldots, 0\leq j_r \leq \frac{D}{a_r}-1$ let, 
by Euclidean division, $\mathfrak q(j_1,\ldots,j_r)$ and $\mathfrak r(j_1,\ldots,j_r)$ be the unique integers such that
\begin{equation}
  a_1j_1+\cdots+a_rj_r=\mathfrak q(j_1,\ldots,j_r)D + \mathfrak r(j_1,\ldots,j_r),\;\; 0\leq \mathfrak r(j_1,\ldots,j_r) \leq D-1.
\end{equation}
We have that $$a_1j_1+\cdots + a_rj_r\leq rD-a_1-\cdots-a_r\leq rD-r<rD,$$ hence $$0\leq \mathfrak q(j_1,\ldots,j_r)\leq r-1.$$



\item We denote the rising factorial by $x^{(r)}:=(x+1)(x+2)\cdots (x+r-1)$, $x^{(0)}=1$. It holds that 
\begin{equation}
\binom{n+r-1}{r-1} = \frac{1}{(r-1)!}n^{(r)} = \frac{1}{(r-1)!} (\stir{r}{r-1} n^{r-1}+\cdots +\stir{r}{1} n + \stir{r}{0}),
\end{equation}
where $\stir{r}{k}$'s are the \emph{unsigned Stirling numbers} of the first kind.


\item Let $w>0$ be a real number. The \emph{Barnes zeta function} associated to $\mathbf a$ and $w$ is
\begin{equation}
\za(s,w):=\sum_{u_1,\ldots,u_r\geq 0}\frac{1}{(a_1u_1+\cdots+a_ru_r+w)^s},\; Re(s)>r.
\end{equation}
For basic properties of the Barnes zeta function see \cite{barnes}, \cite{rui} and \cite{spreafico}.

We have
\begin{equation}
\za(s,w)=\sum_{n=0}^{\infty}\frac{\pa(n)}{(n+w)^s}.
\end{equation}

\item Let 
$$ \zeta_{\mathbf a}(s):=\sum_{n\geq 1}\frac{\pa(n)}{n^s}.$$
Since $\pa(n)=O(n^{r-1})$, the Dirichlet series $\zeta_{\mathbf a}(s)$ is convergent for $Re(s)>r$.
We have
\begin{equation}
\za(s)=\lim_{w\searrow 0}(\za(s,w)-w^{-s}).
\end{equation}
If $r=1$ then 
$$
\zeta_{\mathbf a}(s,w) =  \sum_{n\geq 0}\frac{1}{(a_1n+w)^s} = \frac{1}{a_1^s}\zeta(s,\frac{w}{a_1}),\; \za(s)=\frac{1}{a_1^s}\zeta(s),
$$ 
where $$\zeta(s,w):=\sum_{n=0}^{\infty}\frac{1}{(n+w)^s}, Re(s)>1$$ is the Hurwitz zeta function and 
$$\zeta(s):=\sum_{n=1}^{\infty}\frac{1}{n^s}, Re(s)>1$$ is the Riemann zeta function. 

\end{itemize}


\begin{lema} We have
$$\za(s,w)=\frac{1}{D^s(r-1)!} \sum_{j_1=0}^{\frac{D}{a_1}-1}
 \cdots \sum_{j_r= 0}^{\frac{D}{a_r}-1} \sum_{k=0}^{r-1} \stir{r}{k}
\sum_{j=0}^k  (-1)^j\binom{k}{j}\left(\frac{a_1j_1+\cdots+a_rj_r+w}{D}\right)^j \cdot $$ 
$$ \cdot \zeta(s-k+j,\frac{a_1j_1+\cdots+a_rj_r+w}{D}).$$
\end{lema}

\begin{proof}
Let $u_1\geq 0, u_2\geq 0, \ldots, u_r\geq 0$ be integers. We write $u_1=\frac{D}{a_1} p_1 + j_1$, 
$u_2 = \frac{D}{a_2} p_2 + j_2$, \ldots, $u_r=\frac{D}{a_r}p_r + j_r$, with $0\leq j_1\leq \frac{D}{a_1}-1$, $0\leq j_2 \leq \frac{D}{a_2}-1, \ldots$,
$ 0\leq j_r \leq \frac{D}{a_r}-1$.  We have that 
$$a_1u_1+\cdots+a_ru_r = D (p_1+\cdots+p_r) + a_1j_1+\cdots+a_rj_r.$$ Let $n:=p_1+\cdots+p_r$. 
If $a_1u_1+\cdots+a_ru_r =a_1u'_1+\cdots+a_ru'_r$ with $u'_1=\frac{D}{a_1} p'_1 + j_1$, 
$u'_2 = \frac{D}{a_2} p'_2 + j_2$, \ldots, $u_r=\frac{D}{a_r}p'_r + j_r$ then
 $p'_1+\cdots+p'_r = n$. Therefore, for fixed integers $0\leq j_1\leq \frac{D}{a_1}-1$, 
$0\leq j_2 \leq \frac{D}{a_2}-1, \ldots  0\leq j_r \leq \frac{D}{a_r}-1$ we have in the sum 
$(1.3)$ exactly 
$\binom{n+r-1}{r-1}$ terms $\frac{1}{(a_1u_1+\cdots+a_ru_r+w)^s}$  with  
 $a_1u_1+\cdots+a_ru_r = Dn + a_1j_1+\cdots+a_rj_r$. Hence
$$ \za(s,w)=\sum_{j_1=0}^{\frac{D}{a_1}-1}\sum_{j_2= 0}^{\frac{D}{a_2}-1} \cdots \sum_{j_r= 0}^{\frac{D}{a_r}-1} \sum_{n=0}^{\infty} \frac{\binom{n+r-1}{r-1}}{(Dn + a_1j_1+\cdots+a_rj_r+w)^s} = $$
\begin{equation}
=  \frac{1}{D^s} \sum_{j_1=0}^{\frac{D}{a_1}-1}\sum_{j_2= 0}^{\frac{D}{a_2}-1} \cdots \sum_{j_r= 0}^{\frac{D}{a_r}-1} \sum_{n=0}^{\infty} \frac{\binom{n+r-1}{r-1}}{(n + \frac{a_1j_1+\cdots+a_rj_r+w}{D})^s}. 
\end{equation}

Using the identity $$n^k = (n+\alpha-\alpha)^k = \sum_{j=0}^k (-1)^j\binom{k}{j}(n+\alpha)^j,$$ for $\alpha=\frac{a_1j_1+\cdots+a_rj_r+w}{D}$ and the identity $(1.2)$
the conclusion follows from $(1.6)$.
\end{proof}

\noindent
Note that Bayad and Beck give a decomposition of the Barnes zeta function as a linear combination of Hurwitz zeta functions 
in the case $a_1,\ldots,a_r$ pairwise coprime in \cite[Theorem 1.3]{babeck}.

\begin{lema} We have 
$$\za(s,w)=\frac{1}{D^s(r-1)!} \sum_{j_1=0}^{\frac{D}{a_1}-1}
 \cdots \sum_{j_r= 0}^{\frac{D}{a_r}-1} \sum_{k=0}^{r-1} \stir{r}{k}
\sum_{j=0}^k  (-1)^j\binom{k}{j}\left(\frac{a_1j_1+\cdots+a_rj_r+w}{D}\right)^j \cdot $$ $$\cdot \zeta(s-k+j,\frac{\mathfrak r(j_1,\ldots,j_r)+w}{D}).$$
\end{lema}

\begin{proof}
From $(1.1)$ it follows that {\footnotesize
\[ \zeta(s-k+j,\frac{a_1j_1+\cdots+a_rj_r+w}{D}) = \zeta(s-k+j, \frac{\mathfrak r(j_1,\ldots,j_r)+w}{D}) - 
\sum_{\ell=0}^{\mathfrak q(j_1,\ldots,j_r)-1} \left( \ell + \frac{\mathfrak r(j_1,\ldots,j_r)+w}{D} \right)^{-(s-k+j)}  \]}
\[ = \zeta(s-k+j, \frac{\mathfrak r(j_1,\ldots,j_r)+w}{D}) - \sum_{\ell=1}^{\mathfrak q(j_1,\ldots,j_r)}\left( \frac{a_1j_1+\cdots+a_rj_r+w}{D} - \ell \right)^{-s+k-j}. \]
By Lemma $1.1$, it is enough to show that 
$$
S:=\sum_{j_1=0}^{\frac{D}{a_1}-1}
\sum_{j_2= 0}^{\frac{D}{a_2}-1}
 \cdots \sum_{j_r= 0}^{\frac{D}{a_r}-1}\sum_{k=0}^{r-1} \stir{r}{k} \sum_{j=0}^k (-1)^j\binom{k}{j}\left(\frac{a_1j_1+\cdots+a_rj_r+w}{D}\right)^j 
\cdot $$
$$
\cdot \sum_{\ell=1}^{\mathfrak q(j_1,\ldots,j_r)}\left( \frac{a_1j_1+\cdots+a_rj_r+w}{D} - \ell \right)^{-s+k-j}=0.
$$

Indeed, since {\footnotesize
\[ \sum_{j=0}^k (-1)^j \binom{k}{j} \left(\frac{a_1j_1+\cdots+a_rj_r+w}{D}\right)^j \left( \frac{a_1j_1+\cdots+a_rj_r+w}{D} - \ell \right)^{-s+k-j} = \left( \frac{a_1j_1+\cdots+a_rj_r+w}{D} - \ell \right)^{-s} \cdot \]
\[ \cdot \sum_{j=0}^k  \binom{k}{j} \left(-\frac{a_1j_1+\cdots+a_rj_r+w}{D}\right)^j \left( \frac{a_1j_1+\cdots+a_rj_r+w}{D} - \ell \right)^{k-j} = \left( \frac{a_1j_1+\cdots+a_rj_r+w}{D} - \ell \right)^{-s} (-\ell)^k,\]}
we get 
\[ S= \sum_{j_1=0}^{\frac{D}{a_1}-1}
\sum_{j_2= 0}^{\frac{D}{a_2}-1}
 \cdots \sum_{j_r= 0}^{\frac{D}{a_r}-1}\sum_{k=0}^{r-1} \stir{r}{k} \sum_{\ell=1}^{\mathfrak q(j_1,\ldots,j_r)} \left( \frac{a_1j_1+\cdots+a_rj_r+w}{D} - \ell \right)^{-s} (-\ell)^k =  \]
\[ = \sum_{j_1=0}^{\frac{D}{a_1}-1}
\sum_{j_2= 0}^{\frac{D}{a_2}-1}
 \cdots \sum_{j_r= 0}^{\frac{D}{a_r}-1}\sum_{\ell=1}^{\mathfrak q(j_1,\ldots,j_r)}\left( \frac{a_1j_1+\cdots+a_rj_r+w}{D} - \ell \right)^{-s} 
\cdot \sum_{k=0}^{r-1}\stir{r}{k} (-\ell)^k.  \]
If $\mathfrak q(j_1,\ldots,j_r)=0$ then $\sum_{\ell=1}^{\mathfrak q(j_1,\ldots,j_r)}\left( \frac{a_1j_1+\cdots+a_rj_r+w}{D} - \ell \right)^{-s}=0$, hence  $S=0$. 

If $\mathfrak q(j_1,\ldots,j_r)>0$ then, since $\mathfrak q(j_1,\ldots,j_r)\leq r-1$ we have $1\leq \ell \leq \mathfrak q(j_1,\ldots,j_r)\leq r-1$ so
$\sum_{k=0}^{r-1}\stir{r}{k} (-\ell)^k = (-\ell)^{(r)} = 0$, hence $S=0$.
\end{proof}

\begin{lema}We have \footnotesize{
$$\za(s) = \frac{1}{D^s(r-1)!} \sum_{\substack{0\leq j_1\leq \frac{D}{a_1}-1,\ldots, 0\leq j_r\leq \frac{D}{a_r}-1 \\ 
a_1j_1+\cdots+a_rj_r \not \equiv 0 (\bmod D)}} 
\sum_{j=0}^k \stir{r}{k} (-1)^j\binom{k}{j}\left(\frac{a_1j_1+\cdots+a_rj_r}{D}\right)^j \zeta(s-k+j,\frac{\mathfrak r(j_1,\ldots,j_r)}{D}) + $$}
$$ + \frac{1}{D^s(r-1)!} \sum_{\substack{0\leq j_1\leq \frac{D}{a_1}-1,\ldots, 0\leq j_r\leq \frac{D}{a_r}-1 \\ 
a_1j_1+\cdots+a_rj_r \equiv 0 (\bmod D)}}  \sum_{k=0}^{r-1} 
\sum_{j=0}^k \stir{r}{k} (-1)^j\binom{k}{j}\left(\frac{a_1j_1+\cdots+a_rj_r}{D}\right)^j \zeta(s-k+j).$$
\end{lema}

\begin{proof}
Since $\za(s) = \lim_{w\searrow 0}(\za(s,w)-w^{-s})$, by Lemma $1.2$ it is enough to prove that {\footnotesize
$$ \frac{1}{D^s(r-1)!} \sum_{\substack{0\leq j_1\leq \frac{D}{a_1}-1,\ldots, 0\leq j_r\leq \frac{D}{a_r}-1 \\ 
a_1j_1+\cdots+a_rj_r \equiv 0 (\bmod D)}} \sum_{k=0}^{r-1} 
\sum_{j=0}^k \stir{r}{k} (-1)^j\binom{k}{j}\left(\frac{a_1j_1+\cdots+a_rj_r+w}{D}\right)^j \left(\frac{w}{D}\right)^{-s+k-j} = w^{-s},$$}
which is equivalent to
\begin{equation}
 \sum_{\substack{0\leq j_1\leq \frac{D}{a_1}-1,\ldots, 0\leq j_r\leq \frac{D}{a_r}-1 \\ 
a_1j_1+\cdots+a_rj_r \equiv 0 (\bmod D)}} \sum_{k=0}^{r-1} \stir{r}{k}
\sum_{j=0}^k (-1)^j\binom{k}{j} \left(\frac{a_1j_1+\cdots+a_rj_r+w}{D}\right)^j \left(\frac{w}{D}\right)^{k-j} = (r-1)!.
\end{equation}

Let $0\leq j_1 \leq \frac{D}{a_1}-1,\ldots,0\leq j_r \leq \frac{D}{a_r}-1$ be such that $a_1j_1+\cdots+a_rj_r \equiv 0 (\bmod D)$. 
We have 
$$ \sum_{k=0}^{r-1} \stir{r}{k} \sum_{j=0}^k (-1)^j\binom{k}{j} \left(\frac{a_1j_1+\cdots+a_rj_r+w}{D}\right)^j 
\left(\frac{w}{D}\right)^{k-j} =  $$ $$ =\sum_{k=0}^{r-1} \stir{r}{k} \left(-\frac{a_1j_1+\cdots+a_rj_r}{D}\right)^k 
= \left(-\frac{a_1j_1+\cdots+a_rj_r}{D}\right)^{(r)}.$$
Note that $\frac{a_1j_1+\cdots+a_rj_r}{D}\in \{0,1,\ldots,(r-1)\}$ and $a_1j_1+\cdots+a_rj_r=0$ if and only if $j_1=\cdots=j_r=0$.
Since $0^{(r)}=(r-1)!$ and $(-c)^{(r)}=0$ for $1\leq c\leq r-1$ the formula $(1.7)$ is true.
\end{proof}

\begin{lema} We have 
$$\za(s)=\frac{1}{(r-1)!D^s}\sum_{m=0}^{r-1} \sum_{v=1}^{D} \gamma_{mv} \cdot \zeta(s-m,\frac{v}{D}),$$
where $$ \gamma_{mv} =  \sum_{\substack{0\leq j_1\leq \frac{D}{a_1}-1,\ldots, 0\leq j_r\leq \frac{D}{a_r}-1 \\ 
a_1j_1+\cdots+a_rj_r \equiv 0 (\bmod D)}} \sum_{k=m}^{r-1} \stir{r}{k} (-1)^{k-m} \binom{k}{m} \left(\frac{a_1j_1+\cdots+a_rj_r}{D}\right)^{k-m}.$$
\end{lema}

\begin{proof}
For $0\leq v\leq D-1$, $\mathfrak r(j_1,\ldots,j_r)=v$ if and only if $a_1j_1+\cdots+a_rj_1\equiv v (\bmod D)$.
The conclusion follows immediately from Lemma $1.3$.
\end{proof}

We recall the classical result of Bell \cite{bell}:

\begin{prop}
$\pa(n)$ is a \emph{quasi-polynomial} of degree $r-1$, with the period $D$, i.e. 
$$\pa(n)=d_{\mathbf a,r-1}(n)n^{r-1}+\cdots+d_{\mathbf a,1}(n)n +d_{\mathbf a,0}(n),$$
 where $d_{\mathbf a,m}(n+D)=d_{\mathbf a,m}(n)$ for $0\leq m\leq r-1$ and $n\geq 0$, and $d_{\mathbf a,r-1}(n)$ is not identically zero.
\end{prop}

\begin{lema} For $Re(s)>r$ we have
$$\za(s) = \frac{1}{D^s} \sum_{m=0}^{r-1} \sum_{v=1}^{D} d_{\mathbf a,m}(v)\cdot D^m \cdot \zeta(s-m, \frac{v}{D}).$$
\end{lema}

\begin{proof}
From Proposition $1.5$ it follows that
\[ \za(s)=\sum_{m=0}^{r-1} \sum_{n=1}^{\infty} \frac{d_{\mathbf a,m}(n)}{n^{s-m}} = \sum_{m=0}^{r-1} \sum_{v=1}^{D} d_{\mathbf a,m}(v) \sum_{j=0}^{\infty} 
\frac{1}{(Dj + v)^{s-m}} = \]
\[ = \frac{1}{D^s} \sum_{m=0}^{r-1} \sum_{v=1}^{D}  d_{\mathbf a,m}(v) \cdot D^m \sum_{j=0}^{\infty} 
\frac{1}{(j + \frac{v}{D})^{s-m}} =\frac{1}{D^s} \sum_{m=0}^{r-1} \sum_{v=1}^{D} d_{\mathbf a,m}(v)\cdot D^m \cdot \zeta(s-m, \frac{v}{D}) . \]
\end{proof}

\begin{lema}
Let $k\geq 1$ and $r\geq 0$ be two integers. The set of functions $\{\zeta(s-j,\frac{\ell}{k})\;:\; 0\leq j\leq r,\;1\leq \ell \leq k \}$ 
is linearly independent over $\mathbb C$.
\end{lema}

\begin{proof}
Suppose that $$\sum_{\ell =1}^k \sum_{j=0}^r c_{j\ell} \cdot \zeta(s-j,\frac{\ell}{k}) = 0, \; c_{j\ell}\in \mathbb C,  
0\leq j\leq r,\;1\leq \ell \leq k.$$ 
We want to show that the $c_{j\ell}$'s are $0$. We have 
$$\zeta(s-j,\frac{\ell}{k}) = \sum_{n=0}^{\infty} \frac{1}{(n+\frac{\ell}{k})^{s-j}} = \sum_{n=0}^{\infty} \frac{k^{s-j}}{ (nk+\ell)^{s-j}} 
= \sum_{n=0}^{\infty} \frac{k^{s-j}(nk+\ell)^j}{ (nk+\ell)^{s}}.$$
It follows that
$$ 0=\sum_{\ell =1}^k \sum_{j=0}^r c_{j\ell}\cdot \zeta(s-j,\frac{\ell}{k}) = 
\sum_{\ell =1}^k \sum_{j=0}^r \sum_{n=0}^{\infty} c_{j\ell}\cdot \frac{k^{s-j}(nk+\ell)^j}{ (nk+\ell)^{s}} = $$
$$ = k^s \cdot \sum_{\ell =1}^k \sum_{j=0}^r \sum_{n=0}^{\infty} c_{j\ell}\cdot \frac{k^{-j}(nk+\ell)^j}{ (nk+\ell)^{s}},$$
so $$ \sum_{\ell =1}^k \sum_{j=0}^r \sum_{n=0}^{\infty} c_{j\ell}\cdot \frac{k^{-j}(nk+\ell)^j}{ (nk+\ell)^{s}} = 0.$$
It follows that 
$\sum_{j=0}^r k^{-j} c_{j\ell} (nk+\ell)^j=0$ for every $n\geq 0$, hence  $c_{j\ell}=0$.
\end{proof}

\begin{teor} 
$(1)$ For $0\leq m\leq r-1$ and $n\geq 0$ we have 
$$d_{\mathbf a, m}(n) =\frac{1}{(r-1)!}  \sum_{\substack{0\leq j_1\leq \frac{D}{a_1}-1,\ldots, 0\leq j_r\leq \frac{D}{a_r}-1 \\ a_1j_1+\cdots+a_rj_r \equiv n (\bmod D)}} 
\sum_{k=m}^{r-1} \stir{r}{k} (-1)^{k-m} \binom{k}{m} D^{-k} (a_1j_1 + \cdots + a_rj_r)^{k-m}.$$

$(2)$ We have
$$\pa(n) = \frac{1}{(r-1)!} \sum_{m=0}^{n-1}    \sum_{\substack{0\leq j_1\leq \frac{D}{a_1}-1,\ldots, 0\leq j_r\leq \frac{D}{a_r}-1 \\ 
a_1j_1+\cdots+a_rj_r \equiv n (\bmod D)}} 
\sum_{k=m}^{r-1} \stir{r}{k} (-1)^{k-m} \binom{k}{m} D^{-k} (a_1j_1 + \cdots + a_rj_r)^{k-m}  n^m.$$
\end{teor}

\begin{proof}
$(1)$ From Lemma $1.4$, Lemma $1.6$ and Lemma $1.7$ it follows that
$$d_{\mathbf a,m}(v)D^m = \frac{1}{(r-1)!}\gamma_{mv}$$
for $0\leq m\leq r-1$ and $1\leq v\leq D$. The formulas $(1)$ follow now from Lemma $1.6$ and the fact that the coefficients $d_{\mathbf a,m}(n)$
are periodic with period $D$.

$(2)$ This follows from Proposition $1.5$ and the above formulas $(1)$.
\end{proof}

For $n\geq 1$ we have 
$$p(n)=p_{(1,2,\ldots,n)}(n),$$
where $p(n)$ is the number of partitions of $n$. Applying Theorem $1.8$ for $r=n$ and $\mathbf a=(1,2,\ldots,n)$
we obtain the following formula for $p(n)$:

\begin{cor} 
$$p(n) = \frac{1}{(n-1)!} \sum_{m=0}^{n-1}    \sum_{\substack{0\leq j_1\leq \frac{D_n}{1}-1,\ldots, 0\leq j_n\leq \frac{D_n}{n}-1 \\ j_1+2j_2+\cdots+nj_n \equiv n (\bmod D_n)}} 
\sum_{k=m}^{n-1} \stir{n}{k} (-1)^{k-m} \binom{k}{m} D_n^{-k} (j_1 +2j_2+ \cdots + nj_n)^{k-m}  n^m,$$
where $D_n$ is a common multiple of $1,2,\ldots,n$.
\end{cor}

\begin{cor} We have
$$ \pa(n) = \frac{1}{(r-1)!} \sum_{\substack{0\leq j_1\leq \frac{D}{a_1}-1,\ldots, 0\leq j_r\leq \frac{D}{a_r}-1 \\ 
a_1j_1+\cdots+a_rj_r \equiv n (\bmod D)}} \prod_{\ell=1}^{r-1} (\frac{n-a_1j_{1}- \cdots -a_rj_r}{D}+\ell ).$$
\end{cor}

\begin{proof}
From Theorem $1.8$ it follows that
{\footnotesize
$$ \pa(n)  
=  \sum_{m=0}^{r-1} \frac{1}{(r-1)!} 
\sum_{\substack{0\leq j_1\leq \frac{D}{a_1}-1,\ldots, 0\leq j_r\leq \frac{D}{a_r}-1 \\ 
a_1j_1+\cdots+a_rj_r \equiv n (\bmod D)}} \sum_{k=m}^{r-1} \stir{r}{k}(-1)^{k-m}\binom{k}{m} \left(\frac{a_1j_{1}+\cdots+a_rj_r}{D} \right)^{k-m}
\left(\frac{n}{D}\right)^m = $$}
\begin{equation}
 = \frac{1}{(r-1)!} \sum_{\substack{0\leq j_1\leq \frac{D}{a_1}-1,\ldots, 0\leq j_r\leq \frac{D}{a_r}-1 \\ 
a_1j_1+\cdots+a_rj_r \equiv n (\bmod D)}} \sum_{m=0}^{r-1}\sum_{k=m}^{r-1} \stir{r}{k}\binom{k}{m}
\left(- \frac{a_1j_1+\cdots+a_rj_r}{D} \right)^{k-m}
\left(\frac{n}{D}\right)^m.
\end{equation}
The $m$-th derivative of $x^{(r)}$ is
$$ \frac{d^m x^{(r)}}{dx^m} = m!\cdot \sum_{k=m}^{r-1}\binom{k}{m}\stir{r}{k}x^{k-m}.$$ 
From $(1.8)$ it follows that
$$ \pa(n) = \frac{1}{(r-1)!} \sum_{\substack{0\leq j_1\leq \frac{D}{a_1}-1,\ldots, 0\leq j_r\leq \frac{D}{a_r}-1 \\ 
a_1j_1+\cdots+a_rj_r \equiv n (\bmod D)}} \sum_{m=0}^{r-1} \frac{1}{m!} 
\frac{d^m x^{(r)}}{dx^m}( - \frac{a_1j_1+\cdots+a_rj_r}{D}) \left(\frac{n}{D}\right)^m = $$
$$ =  \frac{1}{(r-1)!} \sum_{\substack{0\leq j_1\leq \frac{D}{a_1}-1,\ldots, 0\leq j_r\leq \frac{D}{a_r}-1 \\ 
a_1j_1+\cdots+a_rj_r \equiv n (\bmod D)}} \left( \frac{n-a_1j_{1}-\cdots-a_rj_r}{D} \right)^{(r)} ,$$
using Taylor's formula.
\end{proof}

\begin{obs}
Let $g:=gcd(a_1,\ldots,a_r)$. For $n\geq 0$ let
$$B_{n}:=\{(j_{1},\ldots,j_r)\;:\; 0\leq j_1 \leq \frac{D}{a_1}-1,\ldots,0\leq j_r \leq \frac{D}{a_r}-1,\; a_1j_{1}+\cdots+a_rj_r \equiv n\;(\bmod D) \}.$$
It holds that
$$\# B_{n} = \begin{cases} g\cdot \frac{D^{r-1}}{a_1 \cdots a_r},& g|n \\ 0,& g\nmid n \end{cases}.$$
Indeed, the map
\[\varphi:  \frac{\mathbb Z}{\frac{D}{a_{1}}\mathbb Z} \times \frac{\mathbb Z}{\frac{D}{a_{2}}\mathbb Z} \times \cdots \times 
\frac{\mathbb Z}{\frac{D}{a_r}\mathbb Z} \longrightarrow \frac{\mathbb Z}{D\mathbb Z}, \]
$$ \varphi(\bar x_{1},\ldots,\bar x_r) := \overline{a_{1}x_{1}+\cdots+a_rx_r} $$
is a group morphism with $\mathtt{Im}(\varphi)=<\bar g>$, the subgroup of $\frac{\mathbb Z}{D \mathbb Z}$ generated by $\bar g$, and,
for $n\geq 0$, the fiber $\varphi^{-1}(\bar n)$ is in bijection with $B_{n}$.

By Corollary $1.10$, in order to determine the restriction partition function $\pa$, we have to compute the fibers of $\varphi$.
We have the following algorithm to compute $\pa(n)$.
\begin{itemize}
\item Input $n\geq 0$. Output $\pa(n)$.
\item If $g\nmid n$, then $\pa(n)=0$. STOP.
\item Let $0\leq m\leq D-1$, $m\equiv n (\bmod D)$. Let $B_m:=\emptyset$. 
\item For each $0\leq j_1 \leq \frac{D}{a_1}-1,\ldots,0\leq j_r \leq \frac{D}{a_r}-1$, if $a_1j_1+\cdots+a_rj_r \equiv m\;(\bmod D)$
then $B_m=B_m\cup \{(j_1,\ldots,j_r)\}$. 
\item $\pa(n) = \frac{1}{(r-1)!} \sum_{(j_1,\ldots,j_r)\in B_{m}} \prod_{\ell=1}^{r-1} (\frac{n-a_1j_{1}- \cdots -a_rj_r}{D}+\ell )$.
\end{itemize}
\end{obs}

Corollary $1.10$ implies the following characterization of the zeros of $\pa(n)$.

\begin{cor}
For $n\geq 0$ we have $\pa(n)=0$ if and only if $n<a_1j_1+\cdots+a_rj_r$ for all 
$0\leq j_1\leq \frac{D}{a_1}-1,\ldots, 0\leq j_r\leq \frac{D}{a_r}-1$ with
$a_1j_1+\cdots+a_rj_r \equiv n (\bmod D)$.
\end{cor}

\begin{proof}
Note that $k^{(r)}=0$ if and only if $k\in\{-r+1,\ldots,-1\}$. 

On the other hand, for all 
$0\leq j_1\leq \frac{D}{a_1}-1,\ldots, 0\leq j_r\leq \frac{D}{a_r}-1$ with
$a_1j_1+\cdots+a_rj_r \equiv n (\bmod D)$ we have
$$\frac{n-a_1j_1+\cdots+a_rj_r}{D}=\frac{n-a_1j_1-\cdots-a_rj_r}{D}\geq \frac{n-rD+a_1+\cdots+a_r}{D} > -r,$$
hence the integer $\frac{n-a_1j_1+\cdots+a_rj_r}{D}$ is $\geq r-1$.
\end{proof}

For $r=2$ and $\mathbf a=(a_1,a_2)$ with $\gcd(a_1,a_2)=1$, Corollary $1.10$ implies the classical result of Popoviciu:

\begin{cor}(Popoviciu \cite[Lemma 11]{popoviciu})
We have
$$\pa(n)= \frac{n+a_1a_1'(n) + a_2a_2'(n)}{a_1a_2} - 1 ,\; n\geq 0,$$
where $a_1'(n)$ and $a_2'(n)$ are the unique integers such that $a_1'(n) a_1 \equiv -n $ mod $a_2$, 
$1\leq a_1'(n)\leq a_2$ and $a_2'(n)a_2 \equiv - n $ mod $a_1$, $1\leq a_2'(n)\leq a_1$. 
\end{cor}

\begin{proof}
Since the map $\varphi: \mathbb Z/a_2\mathbb Z \times \mathbb Z/a_1\mathbb Z \rightarrow \mathbb Z/a_1a_2\mathbb Z$ defined by $\varphi(\hat j_1,\hat j_2):=\overline{a_1j_1+a_2j_2}$
is a group isomorphism, it follows that the congruence $a_1j_1+a_2j_2\equiv n (\bmod a_1a_2)$ has an unique solution $(j_1,j_2)$ with 
$0\leq j_1 \leq a_2-1$ and $0\leq j_2\leq a_1-1$. From Corollary $1.10$, it follows that
$\pa(n)=\frac{n-a_1j_1-a_2j_2}{a_1a_2}+1$. The conclusion follows from the fact that $a'_1(n)=a_1-j_2$ and $a'_2(n)=a_2-j_1$.
\end{proof}

Given a sequence of positive integers $\mathbf a = (a_1,\ldots,a_r)$ with $gcd(a_1,\ldots,a_r)=1$, the \emph{Frobenius number} of $\mathbf a$,
denoted by $F(\mathbf a)=F(a_1,\ldots,a_r)$ is the largest integer $n$ with the property that $\pa(n)=0$.

\begin{cor}
Let $\mathbf a = (a_1,a_2)$ with $\gcd(a_1,a_2)=1$. Then
$F(a_1,a_2)=a_1a_2-a_1-a_2$.
\end{cor}

\begin{proof}
Let $n=a_1a_2-a_1-a_2$. Let $(j_1,j_2)$ be the solution of the congruence $a_1j_1+a_2j_2\equiv n (\bmod a_1a_2)$ with 
$0\leq j_1 \leq a_1-1$ and $0\leq j_2\leq a_2-1$. Assume that $a_1j_1+a_2j_2=a_1a_2-a_1-a_2$. 
It follows that $a_1(j_1+1)+a_2(j_2+1)=a_1a_2$, which is equivalent
to $\frac{j_1+1}{a_2}+\frac{j_2+1}{a_1}=1$, hence we get a contradiction with the fact that $gcd(a_1,a_2)=1$. 
Therefore, by Corollary $1.12$, $\pa(n)=0$.

Now, let $n>a_1a_2-a_1-a_2$. Since $a_1j_1+a_2j_2\leq 2a_1a_2-a_1-a_2$, it follows that $k:=\frac{n-a_1j_1-a_2j_2}{a_1a_2}>-1$ and thus, 
as in the proof of Corollary $1.13$, $\pa(n)=k+1>0$.
\end{proof}

\section{The polynomial part of the restricted partition function}

We recall the following basic facts on quasipolynomials \cite[Proposition 4.4.1]{stanley}.

\begin{prop}
The following conditions on a function $p:\mathbb N \rightarrow \mathbb C$ and integer $D>0$ are equivalent.

(i) $p(n)$ is a quasi-polynomial of period $D$.

(ii) $\sum_{n=0}^{\infty}p(n)z^n = \frac{L(z)}{M(z)}$, where $L(z),M(z)\in \mathbb C[z]$, every zero $\lambda$ of $M(z)$ satisfies $\lambda^D=1$
(provided $\frac{L(z)}{M(z)}$ has been reduced to lowest terms), and $\deg L(z)<\deg M(z)$.

(iii) For all $n\geq 0$, $p(n)=\sum_{\lambda^D=1} P_{\lambda}(n) \lambda^{-n}$, where each $P_{\lambda}(n)$ is a polynomial function.
Moreover, $\deg P_{\lambda}(n) \leq m(\lambda)-1$, where $m(\lambda)$ is the multiplicity of $\lambda$ as a root of $M(z)$.
\end{prop}

Let $p:\mathbb N \rightarrow \mathbb C$ be a quasi-polynomial of degree $r-1 \geq 0$,
$$ p(n):=d_{r-1}(n)n^{r-1}+\cdots+d_1(n)n +d_0(n), $$
where $d_m(n)$'s are periodic functions with integral period $D>0$ and $d_{r-1}(n)$ is not identically zero.
We define the \emph{polynomial part} of $p(n)$ to be the polynomial function $P(n):=P_1(n)$, with the notation from Proposition $2.1(iii)$.

Let $\gamma\in \mathbb C$ with $\gamma^D=1$. It holds that 
$$p_{\gamma}(n):=\gamma^{n}p(n) = \sum_{\lambda^D=1} P_{\lambda}(n) (\gamma\cdot \lambda^{-1})^n,$$
hence $P_{\gamma}(n)$ is the polynomial part of $p_{\gamma}(n)$.

Let $w>0$ be a real number. We consider the function
$$\zeta_p(s,w):=\sum_{n=0}^{\infty}\frac{p(n)}{(n+w)^s}, $$
which is defined for $Re(s)>r$. Similarly, for $\gamma^D=1$ and $Re(s)>r$, we consider
$$\zeta_{p_{\gamma}}(s,w):=\sum_{n=0}^{\infty}\frac{p_{\gamma}(n)}{(n+w)^s}.$$

\begin{prop}
We have 
$$\zeta_p(s,w)= \frac{1}{D^s} \sum_{m=0}^{r-1} \sum_{v=0}^{D-1} d_m(v) \sum_{k=0}^m \binom{m}{k}(-w)^{m-k} D^{k} \zeta(s-k,\frac{v+w}{D}),$$
and therefore $\zeta_p(s,w)$ is a meromorphic function in the whole complex plane with poles at most in the set $\{1,\ldots,r\}$
which are all simple with residues 
$$ R(w,k+1):=\Res_{s=k+1}\zeta_p(s,w) = \frac{1}{D}\sum_{m=k}^{r-1}\binom{m}{k}(-w)^{m-k} \sum_{v=0}^{D-1} d_{m}(v) ,\;0 \leq k\leq r-1.$$
\end{prop}

\begin{proof}
We have that
$$ \zeta_p(s,w) = \sum_{n=0}^{\infty}\frac{p(n)}{(n+w)^s} = \sum_{n=0}^{\infty} \sum_{m=0}^{r-1}d_m(n) \frac{n^m}{(n+w)^s} 
 = \sum_{n=0}^{\infty} \sum_{m=0}^{r-1}d_m(n)  \frac{(n+w-w)^m}{(n+w)^s} = $$ 
$$ = \sum_{n=0}^{\infty} \sum_{m=0}^{r-1}d_m(n) \sum_{k=0}^m \binom{m}{k}(-w)^k\frac{1}{(n+w)^{s-m+k}} 
   = \sum_{m=0}^{r-1} \sum_{v=0}^{D-1} d_m(v) \sum_{k=0}^m \binom{m}{k}(-w)^k \cdot $$ $$ \cdot \sum_{j=0}^{\infty} \frac{1}{(jD+v+w)^{s-m+k}} = \frac{1}{D^s} \sum_{m=0}^{r-1} \sum_{v=0}^{D-1} d_m(v) \sum_{k=0}^m \binom{m}{k}(-w)^k D^{m-k} \zeta(s-m+k,\frac{v+w}{D}) = $$
	$$ = \frac{1}{D^s} \sum_{m=0}^{r-1} \sum_{v=0}^{D-1} d_m(v) \sum_{k=0}^m \binom{m}{k}(-w)^{m-k} D^{k} \zeta(s-k,\frac{v+w}{D}).$$
The second statement is a consequence of the above formula and the fact that $\zeta(s-k,\alpha)$ has a simple pole at $k+1$ with
the residue $1$.
\end{proof}

\begin{cor}
Let $\gamma\in\mathbb C$ with $\gamma^D=1$. We have
$$\zeta_{p_{\gamma}}(s,w)=\frac{1}{D^s} \sum_{m=0}^{r-1} \sum_{v=0}^{D-1}\gamma^{m+v}  d_m(v) \sum_{k=0}^m \binom{m}{k}(-w)^{m-k} D^{k} \zeta(s-k,\frac{v+w}{D}).$$
The function $\zeta_{p_{\gamma}}(s,w)$ is meromorphic in the whole complex plane with poles at most in the set $\{1,\ldots,r\}$
which are all simple with residues 
$$ R_{\gamma}(w,k+1):=\Res_{s=k+1}\zeta_p(s,w) = \frac{1}{D}\sum_{m=k}^{r-1}\binom{m}{k}(-w)^{m-k} \sum_{v=0}^{D-1}\gamma^{m+v} d_{m}(v) ,\;0\leq k\leq r-1.$$
\end{cor}

\begin{proof}
This follows from Proposition $2.2$.
\end{proof}

We consider the Dirichlet series
$$\zeta_p(s):=\sum_{n=1}^{\infty}\frac{p(n)}{n^s},$$
which is convergent for $Re(s)>r$. It holds that
$$\zeta_p(s):=\lim_{w\rightarrow 0}(\zeta_p(s,w)-p(0)\cdot w^{-s}).$$
The Dirichlet series associated to $p_{\gamma}(n)$ is
$$ \zeta_{p_{\gamma}}(s):=\sum_{n=1}^{\infty}\frac{p_{\gamma}(n)}{n^s}.$$
It converges for $Re(s)>r$. It holds that 
$$ \zeta_{p_{\gamma}}(s) = \lim_{w\rightarrow 0}(\zeta_{p_{\gamma}}(s,w)-p_{\gamma}(0)\cdot w^{-s})..$$
As a consequence of Proposition $2.2$ and Corollary $2.3$ we get

\begin{prop}
(i) We have $$\zeta_{p}(s) = \sum_{v=1}^D \sum_{m=0}^{r-1} \frac{1}{D^{s-m}} 
d_m(v)\zeta(s-m,\frac{v}{D}).$$
The function $\zeta_{p}(s)$ is meromorphic in the whole complex plane with poles at most in the set $\{1,\ldots,r\}$
which are all simple with residues $$R_{m}:= \Res_{s=m}\zeta_{p(s)} = \frac{1}{D}\sum_{v=0}^{D-1} d_{m-1}(v),\;1\leq m\leq r.$$

(ii) We have $$\zeta_{p_{\gamma}}(s)=\sum_{v=1}^D \sum_{m=0}^{r-1} \frac{1}{D^{s-m}} 
\gamma^{v}d_m(v)\zeta(s-m,\frac{v}{D}).$$
The function $\zeta_{p_{\gamma}}(s)$ is meromorphic in the whole complex plane with poles at most in the set $\{1,\ldots,r\}$
which are all simple with residues $$R_{\gamma,m}:= \Res_{s=m}\zeta_{p_{\gamma}(s)} = \frac{1}{D}\sum_{v=0}^{D-1} \gamma^{v}d_{m-1}(v),\;1\leq m\leq r.$$

(iii) It holds that $$R(k+1,w)=\sum_{m=k}^{r-1}R_{m+1}\binom{m}{k}(-w)^{m-k},$$
$$ R_{\gamma}(k+1,w)=\sum_{m=k}^{r-1}R_{\gamma,m+1}\binom{m}{k}(-w)^{m-k}.$$
\end{prop}

\begin{prop}
Let $\gamma\in\mathbb C$ with $\gamma^D=1$. We have
$$P_{\gamma}(n)=R_{\gamma,r} n^{r-1} + \cdots + R_{\gamma,2} n + R_{\gamma,1}.$$
In particular, for $\gamma=1$ we have $$P(n)=R_r n^{r-1} + \cdots + R_2 n + R_1.$$ 
\end{prop}

\begin{proof}
Without loss of generality, we may assume that $\gamma=1$ and we prove the last statement.
Let $\alpha_i:=\frac{1}{D}\sum_{v=0}^{D-1} d_i(v)$, $0\leq i\leq r-1$.
Let $\widetilde P(n):=\alpha_d n^d + \cdots + \alpha_1 n + \alpha_0$ and $U(n):=p(n)-\widetilde P(n)$. It follows that
$U(n)=\bar d_{r-1}(n)n^{r-1}+\cdots+\bar d_1(n)n+\bar d_0(n),$
where $\bar d_i(n)=d_i(n)-\alpha_i$, for $0\leq i\leq r-1$. 

Let $S_i(z):=\sum_{n=0}^{\infty} \bar d_i(n) n^i z^n$. It holds that $n^iz^n=\gamma_{ii}(z^{n+i})^{(i)} + \cdots + \gamma_{i1}(z^{n+1})' + \gamma_{i0}z^n,$
for some $\gamma_{ij}\in \mathbb Z$. It follows that
$$ S_i(z) = \sum_{n=0}^{\infty} \bar d_i(n) \sum_{j=0}^i \gamma_{ij} (z^{n+j})^{(j)} = \sum_{v=0}^{D-1} \bar d_i(v) \sum_{m=0}^{\infty}  
\sum_{j=0}^i \gamma_{ij} (z^{mD+v+j})^{(j)} =$$ $$ =  \sum_{v=0}^{D-1} \bar d_i(v) \sum_{j=0}^i \gamma_{ij} \left( \frac{z^{v+j}}{1-z^D} \right)^{(j)}
 =  \sum_{j=0}^i \gamma_{ij} \left(\sum_{v=0}^{D-1} \frac{\bar d_i(v) z^{v+j}}{1-z^D} \right)^{(j)}. $$
Since $$\sum_{v=0}^{D-1} \bar d_i(v)= \sum_{v=0}^{D-1} (d_i(v) -\alpha_i) = \sum_{v=0}^{D-1} d_i(v) - D\alpha_i = 0,$$ 
it follows that $$\sum_{v=0}^{D-1} \frac{\bar d_i(v) z^{v+j}}{1-z^D} = \frac{V_{ij}(z)}{1+z+\cdots+z^{D-1}},$$
 where $V_{ij}(z)$ is a polynomial 
with $\deg(V_{ij}(z))\leq D-2+j$. Therefore $S_i(z)=\frac{L_i(z)}{M_i(z)}$, where $\deg(L_i(z))<\deg(M_i(z))$ and $M_i(1)\neq 0$. 

It follows that
$$ \sum_{n=0}^{\infty} U(n)z^n = \sum_{i=0}^{r-1} S_i(z) = \frac{\overline L(z)}{\overline M(z)}, $$
where $\deg \overline L(z) < \deg \overline M(z)$ and $\overline M(1)\neq 0$.
By \cite[Corollary 4.3.1]{stanley} it follows that 
$$\sum_{n=0}^{\infty}\widetilde P(n)z^n = \frac{\widetilde L(z)}{(1-z)^{r}}, $$
where  $\widetilde L(z)\in \mathbb C[z]$ is a polynomial with $\deg \widetilde L(z) \leq r-1$.

It holds that $$ \sum_{n=0}^{\infty}p(n)z^n = \frac{\widetilde L(z)}{(1-z)^{r}} + \frac{\overline L(z)}{\overline M(z)}, $$
$$ \sum_{n=0}^{\infty} P_1(n)z^n =  \frac{A(z)}{(1-z)^{m+1}}, $$
where $A(z)\in \mathbb C[z]$ is a polynomial of degree $\leq m$,
$$ \sum_{n=0}^{\infty} (p(n)-P_1(n))z^n =\sum_{n=0}^{\infty} \left( \sum_{\lambda^D=1,\; \lambda\neq 1} P_{\lambda}(n)\lambda^n \right) z^n
=\sum_{\lambda^D=1,\; \lambda\neq 1} \sum_{n=0}^{\infty} P_{\lambda}(n)(\lambda z)^n = \frac{B(z)}{C(z)}, $$
where $B(z),C(z)\in \mathbb C[z]$ with $\deg(B(z))<\deg(C(z))$ and $C(1)\neq 0$, hence 
$$\frac{\widetilde L(z)}{(1-z)^{r}} + \frac{\overline L(z)}{\overline M(z)} = \frac{A(z)}{(1-z)^{m+1}}+ \frac{B(z)}{C(z)}.$$
It follows that $\frac{\widetilde L(z)}{(1-z)^{r}} =  \frac{A(z)}{(1-z)^{m+1}}$, so $P_1(n)=\widetilde P(n)$. 
\end{proof}


We return to the setting of Section $1$. As an illustration of our method, we provide new proofs for several results which are known in the literature.
The following corollary extends Theorem $1.1$ of Dilcher and Vignat \cite{dilcher}: the authors prove the formula for $D=a_1\cdots a_r$, and we prove it for any
common multiple of $a_1,\ldots,a_r$.

\begin{cor} For the polynomial part $\Pa(n)$ of the quasi-polynomial $\pa(n)$ we have
$$ \Pa(n) = \frac{1}{D(r-1)!} \sum_{0\leq j_1\leq \frac{D}{a_1}-1,\ldots, 0\leq j_r\leq \frac{D}{a_r}-1} 
\prod_{\ell=1}^{r-1} (\frac{n-a_1j_{1}- \cdots -a_rj_r}{D}+\ell ).$$
\end{cor}

\begin{proof} 
The proof is similar to the proof of Corollary $1.10$, taking into account Proposition $2.4(i)$ and Proposition $2.5$.
\end{proof}

For $1\leq m\leq r$ let $R_m:=Res_{s=m}(\za(s))$.

For $t\geq 0$ let 
$$\alpha_t:=\sum_{0\leq j_1\leq \frac{D}{a_1}-1,\ldots, 0\leq j_r\leq \frac{D}{a_r}-1}(a_1j_1 + \cdots + a_rj_r)^{t}.$$

\begin{cor}
For $1\leq m\leq r$ we have
$$R_m= \frac{1}{D(r-1)!}\sum_{k=m-1}^{r-1} \stir{r}{k} (-1)^{k-m+1} \binom{k}{m-1} D^{-k} \alpha_{k-m+1}. $$
\end{cor}

\begin{proof}
According to Proposition $2.4(i)$ we have
$$R_m=\frac{1}{D}\sum_{v=0}^{D-1} d_{\mathbf a,m-1}(v).$$
The conclusion follows from Theorem $1.8(1)$.
\end{proof}

The \emph{Bernoulli numbers} $B_j$ are defined by
$$ \frac{z}{e^z-1} = \sum_{j=0}^{\infty}B_j \frac{z^j}{j!}, $$
$B_0=1$, $B_1=-\frac{1}{2}$, $B_2=\frac{1}{6}$, $B_4=-\frac{1}{30}$ and $B_n=0$ if $n$ is odd and greater than $1$.

For $k>0$ we have Faulhaber's identity
\begin{equation}
1^k+2^k+\cdots+(n-1)^k = \frac{1}{k+1}\sum_{j=0}^k \binom{k+1}{j} B_j n^{k+1-j}. 
\end{equation}

\begin{lema}
We have
$$ \alpha_{t}=t!\cdot \sum_{i_{1}+\cdots+i_r=t} \sum_{\ell_{1}=0}^{i_{1}} \cdots \sum_{\ell_{r}=0}^{i_{r}}
\frac{B_{\ell_{1}}B_{\ell_{2}}\cdots B_{\ell_r} }{(i_{1}+1-\ell_{1})!\ell_{1}!\cdots (i_{r}+1-\ell_{r})!\ell_{r}! } 
\cdot D^{t+r-\ell_{1}-\cdots -\ell_r} a_{1}^{\ell_{1}-1}\cdots a_r^{\ell_r-1}.$$
\end{lema}

\begin{proof}
We have $$\alpha_t =\sum_{0\leq j_1\leq \frac{D}{a_1}-1,\ldots, 0\leq j_r\leq \frac{D}{a_r}-1}(a_1j_1 + \cdots + a_rj_r)^{t}  = $$
$$= \sum_{0\leq j_1\leq \frac{D}{a_1}-1,\ldots, 0\leq j_r\leq \frac{D}{a_r}-1}\sum_{i_{1}+\cdots+i_r = t}
 \binom{t}{i_{1},\ldots, i_{r}} (a_{1}j_{1})^{i_{1}} \cdots (a_{r}j_{r})^{i_{r}} = $$
\begin{equation}
 =  \sum_{i_{1}+\cdots+i_r = t} \binom{t}{i_{1},\ldots, i_{r}} a_{1}^{i_{1}}\cdots a_{r}^{i_r}\sum_{j_{1}=0}^{\frac{D}{a_{1}}-1}j_{1}^{i_{1}}
\cdots \sum_{j_r=0}^{\frac{D}{a_{r}}-1 }j_r^{i_r} .
\end{equation}
The conclusion follows from $(2.1)$ and $(2.2)$.
\end{proof}

\begin{lema}
For $r\geq 2$ and $0\leq j\leq r-2$ we have
$$ \int_{[0,1]^r} \frac{d^j x^{(r)}}{dx^j}(-t_1-\cdots -t_r)dt_1\cdots dt_r = 0.$$
\end{lema}

\begin{proof}
To ease notation we put $Q(x):=x^{(r)}$. Let $$ I:=\int_{[0,1]^r} Q^{(j)}(-t_1-\cdots -t_r)dt_1\cdots dt_r.$$
Using the change of variables $u_i=t_i-\frac{1}{2}$, $1\leq i\leq r$ we have
$$ I = \int_{[-\frac{1}{2},\frac{1}{2}]^r} Q^{(j)}(-u_1-\cdots -u_r-\frac{r}{2})du_1\cdots du_r, $$
where $Q^{(j)}$ is the $j$-th derivative of $Q$.
Let $\widetilde Q(x):=Q(x-\frac{r}{2})$. If $r$ is even then $\widetilde Q(x)=x(x^2-1^2)(x^2-2^2)\cdots (x^2-(\frac{r}{2}-1)^2)$, hence
 $\widetilde Q(-x)=-\widetilde Q(x)$.
If $r$ is odd then $\widetilde Q(x)=(x^2-(\frac{1}{2})^2) (x^2-(\frac{3}{2})^2) \cdots (x^2-(\frac{r}{2}-1)^2)$, 
hence $\widetilde Q(-x)=\widetilde Q(x)$.
So $\widetilde Q^{(j)}(x)$ is odd if $r-j$ is even and $\widetilde Q^{(j)}(x)$ is even if $r-j$ is odd.

If $r-j$ is even then, using the change of variables $v_i=-u_i$, $1\leq i\leq r$, we have
$$ I = \int_{[-\frac{1}{2},\frac{1}{2}]^r} \widetilde Q^{(j)}(-u_1-\cdots -u_r)du_1\cdots du_r = (-1)^{2r} \int_{[-\frac{1}{2},\frac{1}{2}]^r} \widetilde Q^{(j)}(v_1+ \cdots + v_r)dv_1\cdots dv_r 
 = - I,$$
since $\widetilde Q^{(j)}$ is odd. Thus $I=0$.

If $r-j$ is odd it holds that
$$ I = \int_{[-\frac{1}{2},\frac{1}{2}]^r} \widetilde Q^{(j)}(-u_1-\cdots -u_r)du_1\cdots du_r = \int_{[-\frac{1}{2},\frac{1}{2}]^{r-1}\times [0,\frac{1}{2}]}
 + \int_{[-\frac{1}{2},\frac{1}{2}]^{r-1}\times [-\frac{1}{2},0]}=:I_1+I_2. $$
Using the change of variables $v_i=u_i$, for $1\leq i\leq r-1$ and $v_r=-u_r$ it follows that
$$ I_2 =  \int_{[-\frac{1}{2},\frac{1}{2}]^{r-1}\times [-\frac{1}{2},0]}\widetilde Q^{(j)}(-u_1-\cdots -u_r)du_1\cdots du_r =  $$ $$ =
 - \int_{[-\frac{1}{2},\frac{1}{2}]^{r-1}\times [0,\frac{1}{2}]} \widetilde Q^{(j)}(-v_1-\cdots -v_{r-1}+v_r)dv_1\cdots dv_r.$$
Thus
$$ I = I_1+I_2 = \int_{[-\frac{1}{2},\frac{1}{2}]^{r-1}\times [0,\frac{1}{2}]}( \widetilde Q^{(j)}(-u_1-\cdots -u_r) - \widetilde Q^{(j)}(-u_1-\cdots -u_{r-1}+u_r)) du_1\cdots du_r.$$
Using the change of variables $v_i=-u_i$ for $1\leq i\leq r-1$, $v_r=u_r$ and the fact that $\widetilde Q^{(j)}$ is even it follows that
$$ I =  \int_{[-\frac{1}{2},\frac{1}{2}]^{r-1}\times [0,\frac{1}{2}]}( \widetilde Q^{(j)}(v_1+ \cdots + v_{r-1}-v_r) - \widetilde Q^{(j)}(v_1+ \cdots +v_r)) dv_1\cdots dv_r = -I,$$
hence $I=0$.
\end{proof}

The \emph{Bernoulli--Barnes numbers} $B_j(a_1,\ldots,a_r)$ are defined by
 $$ \frac{z^n}{(e^{a_1z}-1)\cdots (e^{a_rz}-1)} = \sum_{j=0}^{\infty}B_j(a_1,\ldots,a_r) \frac{z^j}{j!}. $$
 The Bernoulli--Barnes numbers and Bernoulli numbers are related as
 \begin{equation}
 B_j(a_1,\ldots,a_r) = \sum_{i_{1}+\cdots+i_r=j} \binom{j}{i_1,\ldots,i_r}B_{i_1}\cdots B_{i_r} a_{1}^{i_{1}-1}\cdots a_r^{i_r-1},
 \end{equation}
 see \cite[Page 2]{babeck}.

The following Theorem can be seen as a
direct consequence of formula $(3.9)$ in Ruijsenaars \cite{rui}. 

\begin{teor}
For $1\leq m \leq r$ we have
$$R_m = \frac{(-1)^{r-m}}{(m-1)!}B_{r-m}(a_1,\ldots,a_r).$$
\end{teor}

\begin{proof}
From Corollary $2.7$ and Lemma $2.8$ it follows that
$$R_m = \frac{1}{D(r-1)!}\sum_{k=m-1}^{r-1} \stir{r}{k} (-1)^{k-m+1} \binom{k}{m-1} D^{-k} \alpha_{k-m+1} =  $$
$$ = \frac{1}{D(r-1)!}\sum_{k=m-1}^{r-1} \stir{r}{k} (-1)^{k-m+1} \frac{k!}{(m-1)!} D^{-k} \cdot $$ {\footnotesize
\begin{equation}
\cdot \sum_{i_{1}+\cdots+i_r=k-m+1} \sum_{\ell_{1}=0}^{i_{1}} \cdots \sum_{\ell_{r}=0}^{i_{r}} 
\frac{B_{\ell_{1}}B_{\ell_{2}}\cdots B_{\ell_r} }{(i_{1}+1-\ell_{1})!\ell_{1}!\cdots (i_{r}+1-\ell_{r})!\ell_{r}! } 
\cdot D^{k-m+1+r-\ell_{1}-\cdots -\ell_r} a_{1}^{\ell_{1}-1}\cdots a_r^{\ell_r-1}.
\end{equation}}
Let $\bar R_m$ be the part of the above sum with $k=r-1$ and $\ell_{1}=i_{1},\ldots,\ell_r=i_r$. We have
$$\bar R_m = \frac{(-1)^{r-m}}{(a_1\cdots a_r)(m-1)!}\sum_{i_{1}+\cdots+i_r=r-m}  \frac{B_{i_1}\cdots B_{i_r}}{i_1!\cdots i_r!} a_1^{i_1} \cdots a_r^{i_r}.$$
We show that $R_m = \bar R_m$. If $r=m$ then $R_m=\bar R_m$. Assume $r\geq 2$ and $m<r$. Let
 $\ell_{1}\geq 0, \ldots, \ell_r\geq 0$ be some integers with $\ell:=\sum_{j=1}^r \ell_j < r-m$. 
The coefficient of $B_{\ell_{1}}B_{\ell_{2}}\cdots B_{\ell_r} D^{r-m-\ell_{1}-\cdots-\ell_r}a_{1}^{\ell_{1}}\cdots a_{r}^{\ell_r}$ in 
$(2.4)$ is {\small
$$S := \sum_{t=\ell}^{r-m} \stir{r}{t+m-1}(-1)^t \frac{(t+m-1)!}{(m-1)!} \sum_{\substack{i_{1}\geq \ell_{1},\ldots,i_r\geq \ell_r \\ i_{1}+\cdots+i_r=t}}  
\frac{1}{(i_{1}+1-\ell_{1})!\ell_{1}!\cdots (i_{r}+1-\ell_{r})!\ell_{r}! }.$$}
Let $s_j=i_j-\ell_j$. It follows that {\small
$$ S= \frac{1}{(m-1)!\ell_{1}! \cdots \ell_r!} \sum_{t=\ell}^{r-m} \stir{r}{t+m-1}(-1)^t \frac{(t+m-1)!}{(t-\ell)!} \sum_{s_{1}+\cdots+s_r = t-\ell}
\frac{\binom{t-\ell}{s_1,\ldots,r_r}}{(s_{1}+1) \cdots (s_r+1)}  =$$
$$ = \frac{(-1)^{\ell}}{(m-1)!\ell_{1}! \cdots \ell_r!} \sum_{t=\ell}^{r-m}\stir{r}{t+m-1}(-1)^{t-\ell} \frac{(t+m-1)!}{(t-\ell)!} 
\int_{[0,1]^{r}} (t_{1}+\cdots+t_r)^{t-\ell}dt_{1}\cdots dt_r = $$}
$$ = \frac{(-1)^{\ell}}{(m-1)!\ell_{1}! \cdots \ell_r!} \int_{[0,1]^{r}} \sum_{t=\ell}^{r-m}\stir{r}{t+m-1} \frac{(t+m-1)!}{(t-\ell)!} 
(-t_{1}-\cdots-t_r)^{t-\ell} dt_{u+1}\cdots dt_r = $$
$$ = \frac{(-1)^{\ell}}{(m-1)!\ell_{1}! \cdots \ell_r!} \int_{[0,1]^{r}} \frac{d^{\ell+m-1} x^{(r)}}{dx^{\ell+m-1}}(-t_{1}-\cdots -t_r) dt_{1}\cdots dt_r = 0,$$
by Lemma $2.9$. Since $R_m - \bar R_m$ is a sum of terms of the form $S$ it follows that $R_m=\bar R_m$. The conclusion follows from
 $(2.3)$.
\end{proof}

\noindent
Bayad and Beck \cite[Formula 1.8, pag. 1323]{babeck} use the Bernoulli-Barnes polynomials defined by
$$ \frac{z^ne^{xz}}{(e^{a_1z}-1)\cdots (e^{a_rz}-1)} = \sum_{j=0}^{\infty}B_j(x;(a_1,\ldots,a_r)) \frac{z^j}{j!}.$$
Our Bernoulli-Barnes numbers $B_j(a_1,\ldots,a_r)$ are related to the Bernoulli-Barnes polynomials by the formula 
$$B_j(a_1,\ldots,a_r)=B_j(0;(a_1,\ldots,a_r)).$$
In the proof of their Theorem $3.1$ Bayad and Beck compute the residue at $z=1$ of the 
function $$F_t(z)=\frac{1}{z^{t+1}\prod_{i=1}^r(1-z^{a_i})}$$ as being $$\frac{(-1)^r}{(r-1)!}B_{r-1}(-t;(a_1,\ldots,a_r)).$$
For $t=0$ one obtains $$\frac{(-1)^r}{(r-1)!}B_{r-1}(0;(a_1,\ldots,a_r)) =  \frac{(-1)^r}{(r-1)!}B_{r-1}(a_1,\ldots,a_r) = -R_1,$$
where $R_1$ is the residue of the function $\za(s)$ at $s=1$, which we haved computed in our Theorem $2.10$.

In the last corollary we provide a new proof for the formula of Beck, Gessler and Komatsu \cite[page 2]{beck} of the polynomial part of $\pa(n)$.

\begin{cor}
The polynomial part of $\pa(n)$ is
$$P_{\mathbf a}(n) := \frac{1}{a_1\cdots a_r}\sum_{u=0}^{r-1}\frac{(-1)^u}{(r-1-u)!}\sum_{i_1+\cdots+i_r=u} 
\frac{B_{i_1}\cdots B_{i_r}}{i_1!\cdots i_r!}a_1^{i_1}\cdots a_r^{i_r} n^{r-1-u}.$$
\end{cor}

\begin{proof}
This follows from Proposition $2.6$ and Theorem $2.10$.
\end{proof}

\section*{Acknowledgement}

We thank the referee for the valuable suggestions which helped to improve our paper.

{}

\vspace{2mm} \noindent {\footnotesize
\begin{minipage}[b]{15cm}
Mircea Cimpoea\c s, Simion Stoilow Institute of Mathematics, Research unit 5, P.O.Box 1-764,\\
Bucharest 014700, Romania, E-mail: mircea.cimpoeas@imar.ro
\end{minipage}}

\vspace{2mm} \noindent {\footnotesize
\begin{minipage}[b]{15cm}
Florin Nicolae, Simion Stoilow Institute of Mathematics, P.O.Box 1-764,\\
Bucharest 014700, Romania, E-mail: florin.nicolae@imar.ro
\end{minipage}}

\end{document}